\documentclass[11pt,oneside]{amsart}
\usepackage{amssymb, amscd, amsmath, amsthm, latexsym, hyperref, bm}
\usepackage{pb-diagram, fancyhdr, graphicx, psfrag,  enumerate, colonequals}
\usepackage[table]{xcolor}
\usepackage[text={6.3in,8.5in},centering,letterpaper,dvips]{geometry}
\usepackage{enumitem}\setlist[enumerate,1]{font=\upshape, itemsep=1ex}\setlist[itemize,1]{font=\upshape, itemsep=1ex}

\def\Z{{\mathbb Z}}

\def\R{{\mathbb R}}

\def\co{\colon\!}

\def\cs{\mathbin{\#}}
\def\inte{\text{int}}

\newcommand{\spinc}{\ifmmode{{\mathfrak s}}\else{${\mathfrak s}$\ }\fi}
\newcommand{\spinct}{\ifmmode{{\mathfrak t}}\else{${\mathfrak t}$\ }\fi}

\newtheorem{theorem}{Theorem}
\newtheorem{lemma}[theorem]{Lemma}
\newtheorem{corollary}[theorem]{Corollary}

\newtheorem{question}[theorem]{Question}

\theoremstyle{definition}

\begin{document}

\title{Connected Sums of Codimension Two Locally Flat Submanifolds}
\author{Charles Livingston}
\address{Charles Livingston: Department of Mathematics, Indiana University, Bloomington, IN 47405 }
\email{livingst@indiana.edu}
\thanks{This work was supported by a grant from the National Science Foundation, NSF-DMS-1505586.}


\begin{abstract}  Let $  X $ and $  Y $  be oriented topological manifolds of dimension $n\!+\!2$, and let $K\!   \subset\! X $ and $J \! \subset\! Y $ be connected, locally flat, oriented, $n$--dimensional submanifolds.   We show that up to orientation preserving homeomorphism there is a well-defined  connected sum $K\! \cs\! J \subset X \! \cs \!Y  $.   For $n = 1$,    the proof is classical, relying on results of Rado and Moise.  For dimensions $n=3$ and  $n \ge 6$, results of  Edwards-Kirby, Kirby, and Kirby-Siebenmann concerning higher dimensional topological manifolds are required.  For $n = 2, 4, $ and $5$, Freedman and Quinn's work on topological four-manifolds is needed.   \end{abstract}

\maketitle
\section{Introduction}     

The proof that the connected sum of $n$--manifolds is well-defined in the topological category is surprisingly deep.  For   $n \ge 6$    it is a consequence of the Annulus Theorem or   Stable Homeomorphism Theorem, proved by Kirby~\cite{MR242165}.  In dimensions  $n\!=\!4$ and $n\!=\!5$   the proof   relies on Freedman and Quinn's work concerning   topological  4--manifolds~\cite{MR1201584,MR679066}.

Proving that connected sums of locally flat $n$--dimensional  submanifolds of manifolds of dimension $n + 2$, or the special case of connected sums of locally flat $n$--knots in $S^{n+2}$, is well-defined  is   more challenging.  Here the  proof relies on the existence and uniqueness of normal bundles in codimension two, results that in turn rely on  the   $s$--cobordism theorem with fundamental group $\Z$. This was proved by Kirby-Siebenmann~\cite[Chapter III, Section 3.4]{MR0645390}  for higher dimensions  and by Quinn~\cite{MR679069} for cobordisms of dimension five.  

Cappell and Shaneson~\cite{MR321099} briefly sketched a proof that   connected sums of knotted $n$--spheres in $S^{n+2}$ are well-defined in dimensions $n\!\ge\!3$.  Our argument roughly follows their approach, but there is a subtlety that seems to have been missed: it appears that the cases of $n=4$ and $n=5$ cannot be proved without results of Freedman-Quinn.    In addition, at the time~\cite{MR321099} was written,  references for topological manifold theory were not yet available.  One purpose of this note is to fill in those gaps. 

We work entirely in the   topological, locally flat, oriented category.  Unless explicitly stated, all manifolds will be assumed to be in that category.   We will show  that if $F_1$ and $F_2$  are  connected,   $n$--dimensional submanifolds of   $(n + 2)$--manifolds $W_1$ and $W_2$, then there is a well-defined connected sum $F_1 \cs F_2 \subset W_1 \cs W_2$ {\color{black} up to homeomorphism of pairs}.  

\smallskip

 \noindent{\bf Outline\ } In Section~\ref{sec:homeo}    we summarize a proof that every orientation preserving homeomorphism  of $S^n$ is isotopic to the identity.  
Section~\ref{sec:conn} presents a proof that the connected sum of $n$--manifolds is well-defined.  For the most part the proof is as would be expected from the proof in the smooth category; an  unexpected feature is the   necessity of  the Stable Homeomorphism Theorem in dimension $n-1$.
In Section~\ref{sec:conn pair} we present the proof that the connected sums of pairs of manifolds of the form  $(W^{n+2}, F^n)$  is well-defined.  Section~\ref{sec:relative} discusses an  alternative proof using a relative form of the Annulus Theorem.

We conclude with an  appendix that provides background and references for some of the key tools used in the argument:  the Alexander Trick, Stable Homeomorphisms, the Annulus Theorem, the Isotopy Extension Theorem, and the existence and uniqueness of normal bundles in codimension two.   We also discuss the proof offered in~\cite{MR321099}.

\smallskip

\noindent{\it Acknowledgements} I am especially grateful to Rob Kirby for discussing with me  the background material that is  based on his work and his joint work with Edwards and with Siebenmann.  One of his suggestions led to a significant simplification of the proof.  Thanks are due to Jim Davis for his repeatedly helping me with other background material.    Mike Freedman and   Kent Orr   also provided  helpful commentary as I prepared this exposition.  Mark Powell read a complete early draft and offered significant improvements, as did Aru Ray.  Shida Wang identified a number of subtle points that had to be addressed.

\section{Homeomorphisms of $S^n$} \label{sec:homeo}

We begin with one of the key background results, that orientation preserving homeomorphisms of $S^n$ are isotopic to the identity.  The proof depends on  some fundamental theorems mentioned above:  the Alexander Trick; the Stable Homeomorphism Theorem for $n\ge 5$,  proved by Kirby in~\cite{MR242165}; and the Stable Homeomorphism Theorem for $n=4$, which is implied by the Annulus Theorem proved by  Quinn~\cite{MR679069}. See Appendices~\ref{app:alex},~\ref{app:annul},~and~\ref{app:stable} for   more details about these three results and relationships between them.
  
 \begin{theorem}\label{thm:iso}  For all $n\ge 0$, every orientation preserving homeomorphism $\phi \co S^n \to S^n$ is isotopic to the identity.
  \end{theorem}

 \begin{proof}  The Stable Homeomorphism Theorem
  states that any orientation preserving homeomorphism of $\R^n$ is {\it stable}, meaning that it is a composition of homeomorphisms, each one of which is the identify on some open subspace.  This quickly implies that orientation preserving homeomorphisms of $S^n$ are similarly stable.  If a homeomorphism $f$ of $S^n$ fixes an open subspace, it fixes an embedded closed ball for which the  complement has closure an embedded ball.  The Alexander Trick permits one to construct an isotopy of that  complementary closed ball, fixing its boundary, that carries $f$   to the identity.

 \end{proof}

\medskip
\noindent{\bf Note.} This theorem does not hold in the smooth setting.  Milnor's first examples of exotic structures on $S^7$ could be built by gluing together two 7--balls via a nonstandard diffeomorphism of $S^6$.  See reference~\cite{MR82103}.  On the other hand, a theorem proved independently by Cerf~\cite{MR140120} and by Palais~\cite{MR116352} states that if an orientation preserving diffeomorphism of $S^n$ extends to a diffeomorphism of  $B^{n+1}$, then it is smoothly isotopic to the identity.  This is all that is needed in the proof that connected sums are well-defined in the smooth category.


\section{Connected sums of manifolds}\label{sec:conn}
In this section we prove that  the connected sum of  $n$--manifolds   is well-defined.  In doing so, we   highlight something  unexpected, the effect of which is that the proof that connected sums of $n$--manifolds is well-defined depends on the Stable Homeomorphism Theorem in dimension  $(n\!-\!1)$ {\color{black} as well as the Annulus Theorem in dimension $n$}. 

 
 We begin with the definition of connected sums.  
Let $W_1$ and $W_2$ be connected oriented $n$--manifolds.  Choose  an orientation preserving embedding  $\phi  \co \R^n \to W_1$ and an orientation reversing embedding $\psi    \co \R^n \to W_2$.    The connected sum is a quotient of  the disjoint union of punctured copies of $W_1$  and $W_2$:
\[
W_1 \cs_{\phi, \psi} W_2 \colonequals \Big(      \big( W_1 \setminus \inte(\phi (B^n)) \big) \   \text{\Large{$\sqcup$}}\     \big( W_2 \setminus \inte(\psi (B^n)) \big) \Big) / \ \! \! \sim\, ,
\]
where the  equivalence relation  identifies $\phi (\theta)  $ with $\psi  (\theta  )$ for $\theta \in S^{n-1}$. We chose embeddings of $\R^n$ to ensure that $\partial \phi(B_n) $ is locally flat.   In addition, if one wishes to write down the finer details of the argument, the map from $\R^n$ provides  a coordinate system in which to work.  One could instead start with embeddings of $B^n$ that have locally flat boundaries, and then apply Brown's theorem~\cite{MR0133812} on the existence of collar neighborhoods of boundaries of manifolds.

\begin{theorem}  The oriented homeomorphism type of $W_1 \cs_{\phi, \psi} W_2  $ is independent of the choice of $\phi$ and $\psi$. 
\end{theorem}

 \begin{proof}

    We prove the  independence on $\phi $.  Let $ {\phi'}$ be a second embedding.  By applying a sequence of homeomorphisms, we modify $\phi'$ in a way that doesn't affect the homeomorphism type of  $W_1 \cs_{\phi', \psi} W_2  $, eventually arriving at $\phi$.
 \begin{enumerate}
 \item  {$\boldsymbol{   \phi' (0) =  {\phi}(0).}$}  There is a homeomorphism $h$ of $W_1$ for which $h( {\phi'}(0)) = \phi (0).$  Thus, we can assume $\phi (0) =  {\phi'}(0)$.  This uses the connectivity of $W_1$.
 
 \item  $\boldsymbol{ \phi'(B^n)  \subset \text{int}(\phi(B^{n}))}.$   For any $\epsilon$ satisfying  $0<\epsilon <1$, there is a homeomorphism $h $ of $\R^n$ that is the identity outside of $2B^n$  and for which $h(B^n) \subset \epsilon B^n$.   This can be used to construct a  homeomorphism $h$ of $W_1$ for which $h(\phi'(B^n)) \subset \text{int}(\phi(B^{n}))$.   
 
 \item  $\boldsymbol{ \phi'(B_n) =   \phi(B_n)}.$  The Annulus Theorem provides coordinates along which $\phi'(B^n)$ can be expanded to match $\phi(B_n)$.  Notice that to do so, one requires tubular neighborhoods of $\partial \phi'(B^n) $ and  $\partial \phi (B^n) $.  As an alternative, this statement can also be expressed in terms of uniqueness of tubular neighborhoods of points, which also relies on the Annulus Theorem.
  
 \item 
  $\boldsymbol{ \text{For all } x \in \partial B^n \cong S^{n-1},   \phi'(x) =     \phi(x).}$  By Theorem~\ref{thm:iso} the restrictions to $\partial B^n$  of the composition  $\phi'\circ \phi^{-1}$ is isotopic to the identity.    We can then use a collar neighborhood of   $\partial B^n$  to extend that isotopy to all of $W_1$.

 \end{enumerate}

 \end{proof}

   \noindent{\bf Observation.}  The map $\phi$ defines a canonical embedding    $\Gamma_{\phi,\psi}\co S^{n-1} \to W_1 \cs_{\phi, \psi} W_2$.   The proof above demonstrates the following.  
   
   \begin{theorem}\label{thm:map}
   The pair $( W_1 \cs_{\phi, \psi} W_2, \Gamma_{\phi, \psi})$    is independent of $\phi$ and $\psi$ up to homeomorphism.
   \end{theorem}
   
   This is a technical extension, but there is reason to distinguish the knot $\Gamma_{\phi,\psi}(S^n)$ from the embedding $\Gamma_{\phi,\psi}$.   For instance, there  exists a smooth embedding  of $S^6$ into $S^8$ with unknotted image, but for which there is no diffeomorphism of $S^8$ carrying the embedding to the standard embedding.  But note that Theorem~\ref{thm:map} does hold in the smooth category.

 
\section{Connected sums of codimension two submanifolds}\label{sec:conn pair}

 Let $n >0$  and suppose that $F_1$ and $F_2$ are $n$--dimensional  oriented, locally flat, connected submanifolds of $(n\!+\!2)$--dimensional  oriented manifolds $W_1$ and $W_2$.  Local flatness ensures that there exists  an   orientation preserving embedding $\phi \co \R^{n+2}  \to W_1$ such that $\phi^{-1}(F_1) = \R^n$. We view such an embedding as a map of pairs: $\phi \co (\R^{n+2}, \R^n)  \to (W_1, F_1)$.  
   Similarly  choose $\psi \co (\R^{n+2}, \R^n) \to (W_2, F_2)$.   We have the unit balls $B^{n+2} \subset R^{n+2}$ and   $B^{n}\subset  \R^n \subset \R^{n+2}$.  The connected sum of the submanifolds  is defined as 
\[
F_1 \cs_{\phi, \psi}  F_2  \colonequals \Big(      \big( W_1 \setminus \inte(\phi (B^{n+2})), F_1 \setminus  
 \inte(\phi (B^n))  \big)\   \text{\Large{$\sqcup$}}\     \big( W_2 \setminus  \text{int}(\psi (B^{n+2}) ), F_2 \setminus  \text{int}(\psi (B^n) )  \big)\   \Big) /   \sim\ .
 \]
As before, the equivalence relation identifies  $\phi ( \theta )  $ with $\psi  (\theta  )$ for $\theta \in S^{n+1}$.    It is straightforward to show that $F_1 \cs_{\phi, \psi}  F_2 \subset W_1 \cs_{\phi, \psi} W_2$ is locally flat.   Our main result is the following.

\begin{theorem}\label{thm:main1}  Given pairs of embeddings, $(\phi_1, \psi_1)$ and  $(\phi_2, \psi_2)$, the manifold  pairs
\[
F_1 \cs_{\phi_1, \psi_1}  F_2 \subset W_1 \cs_{\phi_1, \psi_1} W_2  \text{\ \ and\ \ }   F_1 \cs_{\phi_2, \psi_2}  F_2 \subset W_1 \cs_{\phi_2, \psi_2} W_2
\] 
are oriented homeomorphic.  A homeomorphism can be chosen that identifies the embeddings of the  canonical splitting $(n\!-\!1)$--spheres.
\end{theorem}

 The proof follows readily from three lemmas.  The first is elementary.  The second is the deepest, depending on the existence and uniqueness theorems   of normal bundles of codimension two submanifolds.   The third, though slightly technical, is elementary.  In the second two, we change our perspective, viewing $(\R^{n+2}, \R^n)$ as the pair $(\R^n \times \R^2, \R^n \times \{0\})$.
   
  \begin{lemma} Let  $F \subset W$ be a connected, codimension-two, locally flat submanifold and let  $\phi \co (\R^{n+2}, \R^n) \to (W, F)$  and  $\phi' \co (\R^{n+2}, \R^n) \to (W, F)$ be embeddings. Then there is an orientation preserving {\color{black} self-homeomorphism} of $(W,F)$ that carries  $\phi'$ to an embedding $\phi''  \co (\R^{n+2}, \R^n) \to (W, F)$ for which   $\phi'' (  (B^{n+2}, B^n) )\subset \phi(  (B^{n+2}, B^n))$.
  
 \end{lemma}
  
 {\color{black}  \begin{proof}   The necessity of working with the pair $(W,F)$ rather than simply with $W$ has been missed in previous discussions, so we will provide a few more details here.  The proof follows readily from the following two observations.
   
\begin{enumerate}
   
 \item Let $a$ and $b$ be points on $F$.  Then there is an orientation preserving  homeomorphism $h\co  (W,F) \to   (W,F)$ for which $h(a) = b$.  To prove this, consider the set 
   \[
 B =   \{ x \in F \ \big| \  \text{there exists an }  h \co  (W,F) \to   (W,F)  \text{ for which } h(a) = x  \}.  
   \]
 Working locally, one can prove that $B$ is both open and closed. 

 \item   To ensure that $\phi'' (  (B^{n+2}, B^n) )\subset \phi(  (B^{n+2}, B^n))$ we can again work locally, using the following observation.  Let $U$ be an arbitrary neighborhood of $0 \in \R^{n+2}$.  Then there is a homeomorphism     $ h \co (\R^{n+2} , \R^n ) \to   (\R^{n+2} , \R^n)$ for which:  $h(B^{n+2}) \subset U$   and for which $h(x) = x$ for all $x$ with $\|x\| \ge 2$.
   
     \end{enumerate}

   \end{proof}
 } 
      
\begin{lemma}\label{lem:unique} Let $\phi \co (B^n \times B^2, B^n \times \{0\})   \to  \inte \, ((B^n \times B^2, B^n \times \{0\}) )  $  be an embedding.  Assume that $\phi$ extends to an embedding of an open neighborhood of $B^n \times B^2 \subset \R^{n+2}$.  Then there is an ambient  isotopy of $(\R^n \times \R^2,   \R^n \times \{0\})$ carrying $\phi$ to an embedding $\phi'$ such that $\phi'  ( (B^n \times B^2, B^n \times \{0\}) ) =     (B^n \times B^2, B^n \times \{0\}) $.  Furthermore, the isotopy can be chosen so that $\phi'$ is of the form $\phi'(x,y) = (\phi_1(x), \phi_2(x,y))$.
  \end{lemma}
  
  \begin{proof}  
The Annulus Theorem in dimension $n$  implies that the image $\phi(B^n \times \{0\}) \subset \inte(B^n \times \{0\})$ is isotopic to $B^n \times \{0\}$.  By the isotopy extension theorem, we can thus assume that $\phi(B^n \times \{0\})  =  B^n \times \{0\} $.  The condition that $\phi$ has the extension to a neighborhood in $\R^{n+2}$ then ensures that the  image $\phi(B^n \times B^2)$  forms a normal bundle over $B^n \times \{0\}$, which, by the extension theorem for bundles, is a sub-bundle of  a normal bundle to $\R^n \times \{0\} $ in $\R^{n+2}$.  By the uniqueness theorem for  normal bundles, there is a fiber preserving ambient isotopy carrying one bundle to the other.  Restricting to the image of $\phi$ gives the desired result.
  
\end{proof}

{\color{black} 
We now assume that $\phi \co   (B^n \times B^2, B^n \times \{0\}) \to  (B^n \times B^2, B^n \times \{0\})$ is an orientation preserving homeomorphism of pairs that preserves the product structure in the sense that $\phi$ can be decomposed as $\phi(x,y) = (\phi_1(x), \phi_2(x,y))$ for   functions $\phi_1$ and $\phi_2$.}
    
\begin{lemma} \label{lem:ident} The map $\phi$ is isotopic to the identity as a map of pairs.  \end{lemma}
  
\begin{proof}
Consider $\phi_1 = \phi\big|_{B^n}$.  This is an orientation preserving homeomorphism of $B^n$.  By Theorem~\ref{thm:iso},  the restriction to the boundary $S^{n-1}$ is isotopic to the identity.  By the Alexander trick, this isotopy extends to $B^n$ .  The product structure permits us to extend this isotopy to $B^n \times B^2$, and thus    we can assume that $\phi_1$ is the identity and $\phi$ is of the form   
\[
\phi(x, y) = (x, \phi_2(x,y)).
\]
{\color{black} The function $\phi_2$  defines a map $\psi \co B^n \to \text{Homeo}_+(B^2)$; that is, $\psi(x)(y) = \phi_2(x,y)$.  As  we recall in  Appendix~\ref{app:isob}, there is a  deformation retraction giving $\text{Homeo}_+(B^2) \simeq S^1$. In particular,  $\text{Homeo}_+(B^2) $ is path connected.  Thus $\psi$ is homotopic to   the constant map for which $\psi(x)$ is the identity for all $x$.    This homotopy provides the desired isotopy of $\phi$,  completing the proof. } 
 \end{proof}
 

\section{A relative annulus theorem}\label{sec:relative}

An alternative proof to our main result,~Theorem~\ref{thm:map}, could be based on a relative form of the Annulus Theorem.   We outline a proof.
    
 \begin{theorem}\label{thm:extra} Suppose that $f \co (B^{n+2}, B^{n}) \to \inte\, ( B^{n+2}, B^{n}) $ is an embedding and $f(\partial (B^{n+2}, B^{n}) )$ is locally flat. Then there is a homeomorphism 
\[h\co  \big( (S^{n+1}, S^{n-1}) \times [0,1] \big) \to \big(   (B^{n+2}, B^{n}) \setminus {\rm \inte\, } ( f( (B^{n+2}, B^{n})) \big).
\]
\end{theorem}

\begin{proof}[Proof outline]   We have the  normal bundle $E_1$ to $\big(\R^{n} \setminus \inte (B^n)  \big) \subset \R^{n+2}$.  There is also the normal bundle $E_2 = f(E_2')$, where $E_2'$ is the normal bundle to  $  B^n     \subset \R^{n+2}$.   

With care, these bundles  can be chosen to restrict to give normal bundles to $S^{n-1} $ and $f(S^{n-1})$ in $S^{n+1}$ and $f(S^{n+1})$, respectively.  The union of these can be extended to   form a normal bundle to $\R^n \subset \R^{n+2}$.  Let the restriction to $B^{n} \setminus \inte(f(B^n))$ be denoted $E$.  By decreasing the radius, we can assume that $E \subset B^{n+2}  \setminus \inte(f(B^{n+2}))$.  

{\color{black} 
Let $T$ be a tubular neighborhood of $S^{n-1} \subset S^{n+1}$.  Define 

\[
Z =    \big(  S^{n+1} \times \{0\} \big) \cup \big(T \times [0,1] \big) \cup  \big(  S^{n+1} \times \{1\} \big) \subset \big(  S^{n+1}  \times [0,1] \big). 
\]
We can now use the bundle $E$ to build an embedding $\Phi \co Z \to  (B^{n+2}, B^{n}) \setminus \inte( f(B^{n+2}, B^{n})) $.   The map $\Phi$ is  defined on each of the three sets whose union defines $Z$: on $  \big(  S^{n+1} \times \{0\} \big) $ we use the identity map on the first component;  on  $  \big(  S^{n+1} \times \{1\} \big) $ it is defined using $f$; on $Z$ it is defined using a trivialization of $E$.  One must check that these functions can be adjusted to agree on the overlaps.  The obstruction to finding an isotopy of the overlap maps so that they do agree is determined by an element in   $\pi_{n-1}(\text{Homeo}_+(B^2))$. For $n\ge 3$ this group is 0; for $n = 2$ there is an integer-valued obstruction and a   homological argument must be applied to confirm its vanishing in our situation.  (See Appendix~\ref{app:isob} for a discussion of the homotopy equivalence $\text{Homeo}_+(B^2) \sim S^1$.)
}

The closure of the complement of $Z$ in $    S^{n+1}  \times [0,1]  $ is homeomorphic to $S^1 \times B^{n+1}$.  Thus, to complete the proof by extending $\Phi$, it must be shown that the closure of $ \big(  S^{n+1}  \times [0,1] \big) \setminus \Phi(Z)$ is also homeomorphic to $S^1 \times B^{n+1}$.   

The fact that $Z$ is built by removing a ball pair defined by $f$ from a standard ball pair permits one to use a van Kampen's Theorem argument to show $\pi_1( \big(  S^{n+1}  \times [0,1] \big) \setminus \Phi(Z)) \cong \Z$.  Then we see that  the universal cover of $ \big(  S^{n+1}  \times [0,1] \big) \setminus \Phi(Z)$ embeds in $\R \times B^{n+1}$ and a Mayer-Vietoris argument implies the universal cover is acyclic.  Finally, the Hurewicz Theorem implies that $ \big(  S^{n+1}  \times [0,1] \big) \setminus \Phi(Z)$ is a homotopy circle.

The boundary of $ \big(  S^{n+1}  \times [0,1] \big) \setminus \Phi(Z)$ is homeomorphic to $S^1 \times S^n$.  We can attach a copy of $B^2 \times S^{n}$ to $ \big(  S^{n+1}  \times [0,1] \big) \setminus \Phi(Z)$ to build a homotopy sphere which,  by the  truth of the  Poincar\'e Conjecture,  is homeomorphic to $S^{n+2}$.  Thus, $ \big(  S^{n+1}  \times [0,1] \big) \setminus \Phi(Z)$ is the complement of a knot in $S^{n+2}$ and $ \big(  S^{n+1}  \times [0,1] \big) \setminus \Phi(Z)$ has the  homotopy type of $S^1$.  For dimensions $n\ge 3$, Stallings~\cite{MR149458} proved that such knots are standard.  For $n = 2$, this was proved by Freedman-Quinn~\cite{MR1201584}.  Thus, we have   $ \big(  S^{n+1}  \times [0,1] \big) \setminus \Phi(Z) \cong S^1 \times B^{n+1}$.  

It remains to show that $\Phi$ extends as desired.  This is quickly reduced to a question about extending a map $\Psi$ from $S^1 \times S^{n}$ to $S^1 \times B^{n+1}$.  The fact that the $n$--knot constructed in the last paragraph is standard implies not only that the complement is homeomorphic to $S^1 \times B^{n+1}$, but that the homeomorphism preserves the normal bundle structure of the knot.  Thus, we can assume that $\Psi$ is  map of the form $\Psi(t, x) = (t, \psi(t, x))$ for some $\psi$.  The Alexander trick can now be applied to complete the proof. 

\end{proof}    

\appendix

\section{Background} 

\subsection{Alexander Trick}\label{app:alex}  The  expression   ``Alexander Trick''  refers to the theorem that states that every homeomorphism of $S^{n-1}$ extends to $B^n$  and that 
two homeomorphisms of $B^n$ that agree on $S^{n-1}$ are isotopic.  The    isotopy that is constructed is  constant on $S^{n-1}$.  In the proof, a simple  coning construction extends any homeomorphism from $S^{n-1}$ to $B^n$.  The following   lemma provides the necessary tool to conclude the  proof.  The construction    is the ``trick.''
 
\begin{theorem}  If $f \co B^n \to B^n$ is a homeomorphism that restricts to the identity on $S^{n-1}$, then there is an isotopy $F\co B^n \times [0,1] \to B^n$ from $f$ to the identify such that for all $t$ the map $F(\cdot, t)$ is the identity on $S^{n-1}$. 
\end{theorem}

\begin{proof} The isotopy  $F$ can be written explicitly: \[
F( x, t  ) =
\begin{cases}
tf(x/t),  &\text{if } 0 \le  \big| x  \big| < t\\
x,  &\text{if }  t \le  \big| x \big|  \le  1.

\end{cases}
\]

\end{proof}

\begin{corollary}[Alexander Trick]  If $h_1$ and $h_2$ are homeomorphisms of $B^n$ that agree on $S^{n-1}$, then $h_1$ and $h_2$ are isotopic via an isotopy that fixes $S^{n-1}$.
\end{corollary}
\begin{proof}  Let $F\co  B^n \times [0,1] \to B^n$ be an isotopy from $h_2^{-1} \circ h_1$ to the identity that fixes $S^{n-1}$.   Let $H$ be the isotopy defined by $H(x,t) = h_2(F(x,t))$.  Then $H(x,0) = h_1(x)$ and $H(x,1) = h_2(x)$, as desired.

\end{proof}

\subsection{Homeomorphisms of $B^2$}\label{app:isob} 

There are two natural inclusions:  $S^1 \subset \text{Homeo}_+(S^1)$   and $\text{Homeo}_+(S^1) \subset  \text{Homeo}_+(B^2) $. The first is given by complex multiplication by an element in the unit circle and the second is given by the coning construction.  The following  basic theorem was  
first proved by Kneser~\cite{MR1544816}; see also~\cite{MR321124} for finer details and a discussion at the stack exchange~\cite{stack:320741}.  We present a quick summary.

\begin{theorem} There is a strong deformation retraction $\text{Homeo}_+(B^2)  \to S^1$.  
\end{theorem}

\begin{proof} Given    $\phi \in \text{Homeo}_+(B^2)$, the Alexander Trick provides a canonical isotopy that carries it to a new homeomorphism,  the cone on the map $\phi\big|_{S^1}$.  This yields a  strong deformation retraction from  $\text{Homeo}_+(B^2)$ to  $\text{Homeo}_+(S^1)$.

  Let $h \in \text{Homeo}_+(S^1)$.  The next step is to  define an isotopy $H(x, t) \co S^1 \times [0,1] \to S^1$ from $h$ to the map given by multiplication by $h(1)$.  This is 
achieved by first lifting $h$ to a strictly increasing function of period 1 on $\R$, denoted $\widetilde{h}$ (a homeomorphism),  and then applying the isotopy $\widetilde{H} \co \R \times [0,1] \to \R$ given by 
\[
\widetilde{H}(x,t) = (1-t)\widetilde{h}(x) + tx + t\widetilde{h}(0).  
\]


\end{proof}

\subsection{The Annulus Theorem}\label{app:annul} This theorem  states that for all $n\ge 1$, if $f, g \co S^{n-1}  \to \R^n $ are disjoint  locally flat embeddings, then the compact region bounded by $f(S^{n-1})$ and $g(S^{n-1})$ is homeomorphic to $S^{n-1} \times I$.  This was proved by Kirby~\cite{MR242165} for $n\ge 5$.  In dimension four it was   proved by Quinn~\cite{MR679069}.  See Edwards~\cite{MR780581} for a survey.  

\subsection{Stable Homeomorphism Theorem}\label{app:stable}  This states that every orientation preserving  homeomorphism of $\R^n$ is the composition of homeomorphisms, each one of which is the identity map on some open set.  

An argument   similar to the one used in  the proof of Theorem~\ref{thm:iso}  shows that the truth of the Annulus Theorem for all $k \le n$ implies the Stable Homeomorphism Theorem in dimension $n$; it is also the case that the Stable Homeomorphism Theorem in dimension $n$ implies the Annulus Theorem in dimension $n$.    Such relationships were first identified by Brown and Gluck in a series of three papers~\cite{MR158383, MR158384, MR158385}.     

The Stable Homeomorphism Theorem easily implies that orientation preserving homeomorphisms of $S^n$ are stable.   

For $n\ge 5$, the Stable Homeomorphism Theorem was proved by  Kirby~\cite{MR242165}.  Quinn's proof of the Annulus Theorem for $n = 4$ yields a proof of the Stable Homeomorphism Theorem  in that dimension. 

\subsection{Isotopy Extension Theorem}  This theorem states that if  $C \subset M^n$ is a compact subset and $F \co U \times [0,1]  \to M^n$ is an  isotopy of the inclusion map on some open neighborhood $U$ of $C$,  then there is an  ambient isotopy $H \co M \times [0,1] \to M$ that agrees with $F$ for some open set  $V$ with $C \subset V \subset U$.  This  was proved by Edwards-Kirby~\cite{MR283802} in all dimensions.    
    
 \subsection{Existence and uniqueness of normal bundles}\label{app:normal}  
Let $f \co M  \to N $ be an embedding of manifolds.  Roughly stated, a normal bundle to $f(M)$ consists of an abstract vector bundle $E$ over $M$ and an embedding $h$ of the total space of $E$ into $N$ that agrees with $f$ on the zero-section.  The existence and uniqueness theorem for normal bundles (in codimension 2) has a statement of roughly the following form.

\begin{theorem}[Provisional]  Suppose that $f \co M^k\to N^{k+2}$ represents a codimension two locally flat submanifold, that $C \subset M$ is compact, and that  $U$ is a neighborhood of $C$ in $M$.  Furthermore, assume that $E$ is an embedded normal bundle to $f(U)$.  Then there is a normal bundle $E'$ to $f(M)$ agreeing with $E$ on some neighborhood  $U'$ of $C$, where $U' \subset U$.  Given two such extensions, $E'$ and $E''$, there is an ambient isotopy of $N$ carrying $E'$ to $E''$ that is the identity on some neighborhood $V$ of $C$, with $V \subset U' \cap U''$.
\end{theorem}

To attain such a result, the definition of normal bundle must be made more precise:  For instance,  the Alexander Horned Sphere provides a tubular neighborhood of the origin in $\R^3$ with the property that no ambient isotopy carries it to a standard embedded normal bundle.

In Freedman-Quinn~\cite[Section 9.3]{MR1201584} there is a restriction to {\it extendable normal bundles}.  Let $E$ be a normal bundle with an embedding $h$ of its total space into $N$.  Then $E$ is called an {\it extendable normal bundle} if it has the following property:  for any radial embedding $g$ of $E$ into an open convex disk bundle in a vector bundle $F$ over $M$, the map $h\circ g^{-1}$ extends from $g(E)$ to $F$.  

In Kirby-Siebenmann~\cite{MR0400237} the theory is presented in terms of {\it microbundles}, which are open sets $U \subset N$ containing $f(M)$ with strong retractions $U \to f(M)$.  (For details on microbundles, see Milnor's original paper~\cite{MR161346}.)  As described in~\cite{MR0400237},  the results for microbundles in codimension two can be applied in the setting of vector bundles. This depends on two  results.   First, Kister~\cite{MR180986} proved that a microbundle  always contains a fiber bundle with fibers $\R^k$ and structure group $TOP(k)$,  the group of homeomorphisms of $\R^k$ that fix the origin.   Second,  in dimension two, Kneser proved that $TOP(2)$ deformation retracts to $S^1 \cong O(2)$; see~\cite{MR321124,MR1544816}, or the proof presented in Appendex~\ref{app:isob}


\section{Cappell and Shaneson's proof}\label{app:cs} Cappell and Shaneson's proof that connected sums of knots  $\Sigma^n \subset S^{n+2}$ depends on a lemma, the proof of which consists of four sentences.   Slightly changing the notation, it states the following.

\begin{lemma} Let $\Sigma^n$ be a locally flat oriented knot in $S^{n+2}$.  Let $f_i\co (B^{n+2}, B^n) \to (S^{n+2}, \Sigma^n)$ be orientation preserving embeddings for $i = 1, 2$.  Then for $n = 3$ and $n\ge 5$ there is an orientation preserving homeomorphism $G$ of $S^{n+2}$ such that $G\circ f_1 = f_2$.  For $n=4$, the result remains true if   both $f_i$ restrict to give stable maps on $B^4$.

\end{lemma} 
In the first  line of their proof, they let $g_i = f_i\big|_{B^n}$.  The main steps are then as follows.
 
 \begin{enumerate}
 \item  The first  step is the claim that $g_2$ is isotopic to $g_1$ and thus,  by using  the isotopy extension theorem, $f_2$ can be modified so that $g_1 = g_2$.    
It isn't mentioned that  one has to  arrange that the isotopy is of the pair $(S^{n+2}, \Sigma^n)$ and not simply of $S^{n+2}$.  This modification is straightforward.  
 
It is not difficult  to arrange that $g_1(B^n) \subset g_2(B^n)$.  To arrange that $g_1(B^n) = g_2(B^n)$, one can use the Annulus Theorem in dimension $n$.  Again, use the isotopy extension theorem to extend this isotopy to $S^{n+2}$.  However, the claim that $g_1 = g_2$, rather than just having the same images,  $g_1(B^n) = g_2(B^n)$, is not justified.  One must use the fact that orientation preserving homeomorphisms of $B^n$ are isotopic, a fact that depends on the Stable Homeomorphism  Theorem in dimension $n-1$.  

In the case that $n=4$, it is not clear precisely what is meant by the map being {\it stable}.  Such a notion could be made well-defined in the setting that a knot is viewed as an embedding rather than as a {\it submanifold}; without the Stable Homeormorphism Theorem in dimension 4, it was possible that two locally flat embeddings of $S^4$ in $S^6$ could have the same oriented image but be inequivalent as knots.   In the foundational papers on the general theory of stable mappings and manifolds, written by Brown and Gluck~\cite{MR158383, MR158384, MR158385}, the notion of stable maps concerns open subsets of $S^n$.  Thus, it seems that to be precise, one would want to have the $f_i$ defined on open sets $U_i \subset \R^{n+2}$ containing $B^n$.  In that case, the connected sum could depend on the choice of extensions of the $f_i$ from $B^{n+2}$ to $U_i$.  The Annulus Theorem in dimension four dispenses with such issues.
 
 \item The next step is to extend the natural 2--disk bundles over $g_1(B^n)$ that are determined by $f_1$ and $f_2$ to be bundles over $\Sigma^n$. Using this, it is claimed that uniqueness of normal bundles permits one to ensure that $f_1 = f_2$.  However, uniqueness only ensures that the images of the bundle maps (on each fiber) are the same.  It is possible that the two maps differ by a bundle automorphism.  This issue was addressed in Lemma~\ref{lem:ident}.
 
 \end{enumerate}

  
  \section{Connected sums in codimension greater than two}
  
  If one replaces the set-up of Theorem~\ref{thm:main1} by making the single change that the submanifolds are of codimension $k \ge 3$, one can ask whether the same results  holds.  The answer is not evident.
  
\begin{question}  Is the following statement true? Given pairs of embeddings, $(\phi_1, \psi_1)$ and  $(\phi_2, \psi_2)$ of $(\R^{n}, \R^{n-k})$ into $(W_i, F_i)$, the manifold  pairs $F_1 \cs_{\phi_1, \psi_1}  F_2 \subset W_1 \cs W_2 $ and $F_1 \cs_{\phi_2, \psi_2}  F_2 \subset W_1 \cs W_2$ are oriented homeomorphic.  A homeomorphism can be chosen that identifies the canonical splitting $(n\! -\! 1)$--spheres.
\end{question}
    
\noindent{\bf Comments.}  For $k = n-1$ the result is probably easily proved.  For other $k$ it is not at all clear.  On the one hand, the tools used in the codimension two case, such as the existence and uniqueness of normal bundles, are lacking.  On the other hand, our constructions are largely local and in higher codimension it seems there cannot be any local complexity; for instance, all knots of codimension $k \ge 3$ in $S^n$ are trivial by a result of  Stallings~\cite{MR149458}.
\bibliography{../BibTexComplete.bib}
\bibliographystyle{plain}	

\end{document}